\documentclass[12pt]{amsart}
\usepackage[active]{srcltx}
\usepackage{a4wide,amssymb}
\usepackage{graphicx,amssymb,ifthen,url}
\usepackage{algorithm,algorithmic,color,amsmath,epsf,cancel,mathdots}
\usepackage[arrow, matrix, curve]{xy}

\xyoption{all}

\newtheorem{theorem}{Theorem}[section]
\newtheorem{lemma}[theorem]{Lemma}
\newtheorem{proposition}[theorem]{Proposition}
\newtheorem{corollary}[theorem]{Corollary}
\theoremstyle{definition}
\newtheorem{definition}[theorem]{Definition}

\newtheorem{example}[theorem]{Example}

\theoremstyle{remark}
\newtheorem{remark}[theorem]{Remark}

\numberwithin{equation}{section}

\newcommand {\D}  {{\mathcal D}}
\newcommand {\E}  {{\mathcal E}}
\newcommand {\F}  {{\mathcal F}}

\newcommand {\M}  {{\mathcal M}}
\newcommand {\N}  {{\mathcal N}}
\newcommand {\PP}  {{\mathcal P}}
\newcommand {\R}  {{\mathcal R}}

\newcommand {\SSS}  {{\mathcal S}}

\newcommand {\V}  {{\mathcal V}}

\newcommand {\FF}  {{\mathbb F}}

\newcommand {\NN}  {{\mathbb N}}

\newcommand {\RR}  {{\mathbb R}}
\newcommand {\ZZ}  {{\mathbb Z}}

\newcommand{\abs}[1]{\lvert#1\rvert}

\newcommand{\fl}[1]{\left\lfloor #1 \right\rfloor}

\newcommand{\comment}[1]{{}}

\newcommand{\qmboxq}[1]{\quad\mbox{#1}\quad}

\newenvironment{romanlist}
  {%
   \setlength{\topsep}{0pt}%
   \vspace{-\parskip}%
   \begin{enumerate}%
     \setlength{\parsep}{0pt}
     \setlength{\parskip}{0pt}%
  }%
  {\end{enumerate}%
   \vspace{-\parskip}}

\def\eps{\varepsilon}

\def\oddots{\mathinner{\mkern1mu\raise\p@ \vbox{\kern7\p@\hbox{.}}\mkern2mu \raise4\p@\hbox{.}\mkern2mu\raise7\p@\hbox{.}\mkern1mu}}

\providecommand{\sVert}[1][-1]{ \ensuremath{\mathinner{
\ifthenelse{\equal{#1}{-1}}{ 
\rvert}{}
\ifthenelse{\equal{#1}{0}}{ 
\rvert}{}
\ifthenelse{\equal{#1}{1}}{ 
\bigr\rvert}{}
\ifthenelse{\equal{#1}{2}}{ 
\Bigr\rvert}{}
\ifthenelse{\equal{#1}{3}}{ 
\biggr\rvert}{}
\ifthenelse{\equal{#1}{4}}{ 
\Biggr\rvert}{}
}} 
}

\title[Digit Systems over commutative rings]{Digit systems over commutative rings}

\author[K.~Scheicher]{K.~Scheicher}
\address{K.~Scheicher, Institute of Mathematics, University of Natural Resources and Applied Life Sciences, Gregor-Mendel-Stra\ss{}e
33, 1180 Vienna, AUSTRIA} \email{klaus.scheicher@boku.ac.at}
\author[P.~Surer]{P.~Surer}
\address{P.~Surer, Chair of Mathematics and Statistics, University of Leoben, Franz-Josef-Stra\ss{}e 18, 8700 Leoben, AUSTRIA}
\email{me@palovsky.com}
\author[J.~M.~Thuswaldner]{J.~M.~Thuswaldner}
\address{J.~M.~Thuswaldner, Chair of Mathematics and Statistics, University of Leoben, Franz-Josef-Stra\ss{}e 18, 8700 Leoben, AUSTRIA}
\email{Joerg.Thuswaldner@unileoben.ac.at}
\author[C.~E.~van de Woestijne]{C.~E.~van de Woestijne}
\address{C.~E.~van de Woestijne, Institute of Mathematics B,
Technical University of Graz, Steyrergasse 30, 8010 Graz, AUSTRIA}
\email{c.vandewoestijne@tugraz.at}
\thanks{This research was supported by the Austrian Science Foundation (FWF),
projects S9606 and S9610, which are part of the national research network
FWF-S96 ``Analytic combinatorics and probabilistic number
theory''.}

\date{September 2, 2009}

\keywords{Canonical number system, shift radix system, digit systems}

\subjclass[2000]{11A63, 28A80, 52C22}

\begin{document}

\begin{abstract}
Let $\E$ be a commutative ring with identity and $P\in\E[x]$ be a
polynomial. In the present paper we consider digit representations
in the residue class ring $\E[x]/(P)$. In particular, we are
interested in the question whether each $A\in\E[x]/(P)$ can be
represented modulo $P$ in the form $e_0+e_1 X + \cdots + e_h X^h$,
where the $e_i\in\E[x]/(P)$ are taken from a fixed finite set of
digits. This general concept generalises both canonical number
systems and digit systems over finite fields. Due to the fact that
we do not assume that $0$ is an element of the digit set and that
$P$ need not be monic, several new phenomena occur in this context.
\end{abstract}

\maketitle

\begin{section}{Introduction}

In recent years, many different notions of number systems have
been invented and thoroughly studied (see {\it e.g.}
\cite{Allouche1997,BBLT2006} and the references therein). Several
of them, like canonical number systems and digit systems over
finite fields \cite{Petho:91,Scheicher-Thuswaldner:03a} represent
elements of a factor ring of the shape $\E[x]/(P)$, where $\E$ is
a commutative ring with identity and $P\in \E[x]$ is a polynomial.
The representations obtained in these number systems have the
shape $e_0+e_1 X+\cdots + e_h X^h$ where $X$ is the coset of $x$
modulo $P$, and the elements $e_i$ ($0\le i \le h$) are taken from
a finite set $\N \subset \E[x]/(P)$ of digits.

The aim of the present paper is to establish a common general
framework for all number systems of this kind. Indeed, we allow $\E$
to be an arbitrary commutative ring with identity. The only
requirement for the digit set $\N$ is that it has to be a system of
coset representatives of the ring $\E[x]/(P)$ modulo the basis $X$.
(Note that $\E[x]/(P,x)$ is isomorphic to $\E/(P(0))$; thus this
property is decided by looking at the constant coefficients of the
digits alone.) In particular, we allow digit sets $\N$ that do not
contain zero.

Although we are able to prove several theorems in this general
context, new and interesting phenomena and difficulties occur
throughout.

As a first point, treated in Section~\ref{sec1}, if zero is not
contained in the digit set, one has to be careful how to define
finite expansions; also the well known fact that each element of
$\E[x]/(P)$ has a finite representation if and only if the dynamical
system associated to the number system has only zero as a periodic
point, is no longer true. In the general case, there may well occur
a cycle of length greater than $1$ without the finite expansion
property being violated.

Moreover, as we do not have the concept of an expanding polynomial
in our general setting, one has to take care to adopt the right
definition of periodicity. Indeed, let $T$ be the dynamical system
associated to the number system. We can easily show that an
element with eventually periodic orbit admits a digit
representation which is eventually periodic. However, the converse
is only known when the defining polynomial $P$ is expanding.

Secondly, in canonical number systems as well as in digit systems
over finite fields, \emph{monic} defining polynomials $P$ have
been considered almost exclusively (although
\cite{Akiyama-Frougny-Sakarovitch:08}, \cite{Gilbert:81} and
\cite[Section 5.3.3]{Woestijne:PhD:06} do treat nonintegral bases
for number systems in $\ZZ$ from various viewpoints). In
Section~\ref{sec2}, we explore the case of non-monic polynomials.
Here the basis of the number system, taken as a root of the
defining polynomial, need no longer be an algebraic integer. While
$\E[x]/(P)$ is a finitely generated free $\E$ module if $P$ is
monic, neither property holds if $P$ is not monic. Interestingly,
we are able to exhibit a finitely generated (and in many cases
even free) module $\R_k \subset \E[x]/(P)$ with the following
property. If each element of $\R_k$ admits a finite expansion then
the same is true for all elements of $\E[x]/(P)$. This reduces the
problem from a set with complicated algebraic structure to an
easier framework.

In Section~\ref{sec3}, given the exact sequence
$$
  0\rightarrow \E[x]/(P_2) \xrightarrow{\cdot P_1} \E[x]/(P_1P_2) \rightarrow
  \E[x]/(P_1) \rightarrow 0,
$$
of $\E$-modules, with number systems on the outer components, we
construct a number system on the middle module and derive conditions
when all elements in this number system have a finite expansion, or
at least a periodic digit sequence.

In Section~\ref{sec4}, we extend the definitions of canonical number
systems and number systems over finite fields to possibly non-monic
defining polynomials. Interestingly, for canonical number systems,
this more general version is still covered by the theory of shift
radix systems (in the sense of
\cite{Akiyama-Borbeli-Brunotte-Pethoe-Thuswaldner:05}).

In the final Section~\ref{sec5}, we take the concept of a set of
witnesses (due to Brunotte~\cite{Brunotte:01} in the context of
canonical number systems) to our general setting. The idea is that
the finite expansion property of a number system can be decided by
looking only at the elements of a properly chosen subset of $\R$.
Whenever a number system possesses a \emph{finite} set of
witnesses, we can decide the finite expansion property by a finite
algorithm. We prove the existence of finite witness sets for a
large class of number systems.

Throughout the paper we illustrate our concepts with examples which
show the difficulties and new phenomena occurring here.
\end{section}

\begin{section}{Basic properties of digit systems} \label{sec1}

We start this section with a formal definition of a quite general
notion of number system. Fix a commutative ring $\E$ with identity.

\begin{definition} \label{DefBasic}
Let $d \geq 1$ be an integer and
\[P(x)=p_dx^d+\cdots+p_1x+p_0\]
a polynomial with coefficients in $\E$, such that $p_d$ and $p_0$
are not zero divisors of $\E$, and $\E/(p_0)$ is finite.
Furthermore, let $\R = \E[x]/(P)$, and let $\N \subset \R$ be a
system of coset representatives of $\R/(X)$. The {\it digit system
over $\E$ defined by $P$ and $\N$} is the triple $(\R,X,\N)$, where
$X$ is the image of $x$ under the canonical epimorphism $\E[x]
\rightarrow \E[x]/(P)$.

Define the maps
\[\begin{split}
  D_\N &: \R\rightarrow \N : D_\N(A) = e\text{, the unique $e\in\N$ with } A\equiv e\pmod{X} , \\
  T &: \R \rightarrow \R : T(A) = \frac{A-D_\N(A)}X.
\end{split}\]
We say that $(\R,X,\N)$ has the {\it periodic expansion property}
(PEP) if the sequence $(T^i(A))_{i\geq0}$ is eventually periodic for
each $A\in \R$. We call the sequence $(D_\N(T^i(A)))_{i \geq 0}$ the
{\it digit sequence of $A$ in the digit system $(\R,X,\N)$}.
\end{definition}
\begin{lemma}\label{pepperseq}
The PEP implies that each $A\in\R$ has an eventually periodic digit
sequence.
\end{lemma}
\begin{proof}Trivial.
\end{proof}
It is not clear to us whether the converse also holds; it holds
trivially whenever the \emph{periodic set} of the number system
(cf.\ Definition~\ref{DefPeriodicSet}) is finite.

One notes that $\R/(X)\cong \E/(p_0)$, so that $|\N|=|\E/(p_0)|$.
The digit sequence of $A\in\R$ clearly exists and is unique, because
$\N$ is a system of representatives of $\R$ modulo $X$. If $\N$ were
larger, we would have non-uniqueness and nondeterminism in the
$X$-ary representation of $A$. The repeated application of the map
$T$ gives the backward division algorithm, as defined in
\cite{Gilbert:81} and many later papers.

\begin{lemma}\label{lemma1}
There exists $n\in\NN$ such that
\begin{equation}\label{rep1}
A = \sum_{i=0}^{n-1} D_\N(T^i(A)) X^i
\end{equation}
if and only if $T^n(A)=0$ for some $n \in \NN$.
\end{lemma}

\begin{proof}
Assume that $A$ has the representation~\eqref{rep1}.
We obviously
have $T(A) = \frac{A-D_\N(A)}X = \sum_{i=0}^{n-2} D_\N(T^{i+1}(A)) X^i$.
Continuing, by induction, we obtain $T^{n}(A)=0$. Conversely,
suppose there exists an $n \in \NN$ such that $T^n(A)=0$. Then it is easy to see that
$A =\sum_{i=0}^{n-1} D_\N(T^{i}(A)) X^i$.
\end{proof}

Lemma~\ref{lemma1} motivates the following definition.
\begin{definition}
Let $(\R,X,\N)$ be a digit system over $\E$. $(\R,X,\N)$ has the {\it finite expansion
property} (FEP) if for each element $A \in \R$ there exists an $n \in \NN$ with $T^n(A)=0$.
The representation \eqref{rep1} is called the finite $X$-ary expansion of $A$.
\end{definition}

\begin{example}\label{ex1}
Let $\E=\ZZ$, $P(x)=3x^2-2x+5$, and $\N=\{0,1,2,3,4\}$. We want to
calculate the digit  sequence of $-X^k$ for $k \geq 0$ (where
$X^0=1$) in the digit system $(\ZZ[x]/(P),X,\N)$. The difficulty in
the non-monic case is that we cannot represent $-X^k$ in
$\ZZ[x]/(P)$ as a linear combination of smaller powers of $X$. We
have $-1 \equiv 4 \pmod{X}$ and therefore
\[T(-1)=\frac{-1- D_\N(-1)}{X} = \frac{-1- 4}{X} = \frac{3X^2-2X}{X} = 3X-2.\]
Continuing in this way we obtain
\[
\begin{split}
T^2(-1) & = T(3X-2)=\frac{3X-2- D_\N(3X-2)}{X} = \frac{3X-2- 3}{X} = \frac{3X^2+X}{X} = 3X+1 \\
T^3(-1) & = T(3X+1)= 3  \\
T^4(-1) & = T(3)= 0.
\end{split}
\]
Thus we have $(D_\N(T^i(-1)))_{i \geq 0}=4,3,1,3,0,0,\ldots$ and obviously
\[(D_\N(T^i(-X^k)))_{i  \geq 0}=\underbrace{0,\ldots,0}_{k},4,3,1,3,0,0,\ldots.\]
We see that the element $-X^k$ has the finite $X$-ary expansion
\[-X^k=\sum_{j=0}^{k+3} D_\N(T^j(-X^k)) X^j = 4X^k+3X^{k+1}+X^{k+2}+3X^{k+3}.\]

Alternatively, if we take $\M=\{-2,-1,0,1,2\}$ as digits, we easily obtain
\[(D_\M(T^i(-X^k)))_{i \geq 0}=\underbrace{0,\ldots,0}_{k},-1,0,0,\ldots.\]
The finite $X$-ary expansion of $-X^k$ is just $-1X^k$.

We are going to deal with such generalisations of canonical number
systems  in Subsection~\ref{CNS}. There we also will see whether
$(\R,X,\N)$ and $(\R,X,\M)$, respectively, have the PEP or FEP.
\end{example}

\begin{lemma} \label{lemma2}
The finite expansion property implies the periodic expansion property.
\end{lemma}

\begin{proof}
Suppose that there exists an $A$ having finite $X$-ary expansion and
no periodic digit sequence. Then there exists an $n \in \NN$ with
$T^n(A)=0$. Since $A$ is not periodic we have $T^k(A)\not=0$ for all
$k>n$. Thus the element $T^{n+1}(A)=T(0)\ne 0$ cannot have a finite
$X$-ary  expansion.
\end{proof}

\begin{definition} \label{DefZeroExpansion}
  A \emph{zero cycle} of a digit system $(\R,X,\N)$ is a finite sequence
  $(d_0,\ldots,d_\ell)$, with $d_i\in\N$ and $\ell\ge 0$, such that
  \begin{equation}\label{zeroexp}
    0=\sum_{i=0}^\ell d_iX^i.
  \end{equation}
  The \emph{zero period} of $(\R,X,\N)$ is the length of a shortest
  zero cycle, if such a one exists, and undefined otherwise.
\end{definition}

Note that the finite expansion \eqref{zeroexp} of $0$ given by a
zero cycle is different from the trivial expansion of $0$ by an
empty sum.
\begin{lemma}
The finite expansion property implies the existence of a zero cycle.
\end{lemma}
\begin{proof}
If $0\in\N$, the assertion is trivial, because $(0)$ is a zero
cycle.  Otherwise, let $a=T(0)$, so that $aX + D_\N(0)=0$. We have
$a\ne 0$; by assumption, there is a finite expansion $a =
\sum_{i=0}^{\ell} d_iX^i$ $(d_i\in\N)$. But then
$$
  0 = \sum_{i=0}^\ell d_iX^{i+1} + D_\N(0).
$$
Thus $(D_\N(0),d_0,\ldots,d_{\ell})$ is a zero cycle.
\end{proof}

Note that zero cycles are uniquely
determined by their length; under the periodic expansion property, if
a zero cycle exists, it is a concatenation of copies of the shortest zero
cycle.

Every finite expansion of an element $A \in \R$ can be prolonged
indefinitely by appending the sum~\eqref{zeroexp} corresponding to
the zero cycle; this generalises padding with zeros in case $0 \in
\N$. This gives another (constructive) proof of Lemma~\ref{lemma2}.

\begin{definition} \label{DefPeriodicSet}
  The \emph{periodic set} $\PP$ of a digit system $(\R,X,\N)$ is
  the set of all elements $A$ of $\R$ with $T^n(A)=A$ for some $n \geq 1$.
\end{definition}

$\PP$ is thus the set of all elements of $(\R,X,\N)$ that are purely
periodic under the action of $T$. The map $T$ permutes $\PP$, and we
can consider the quotient $\PP/T$, which is the set of orbits in
$\PP$ under the action of $T$. Note that the orbits are finite by
definition. Clearly, $\R$ has the PEP if and only if for all
$A\in\R$ there exists some $n$ with $T^n(A)\in\PP$.

In general, it is not clear if $\PP$ is, for example, nonempty, a
singleton, or finite. However, in the special case where $\R$ can be
embedded in a finite-dimensional complex vector space, we can
consider the {\it expanding property} on the defining polynomial
$P$, which requires that all zeros have modulus strictly greater
than $1$.  This property at once implies the nonemptiness and
finiteness of $\PP$, as well as the periodic expansion property (cf.
Section~\ref{CNS}). In this case, $\PP$ is also called the {\it
attractor} of $(\R,X,\N)$.

The next lemma gives another criterium for the finiteness of $\PP$.
The essence of this result is well-known, cf.\ for example
\cite[Lemma 2.1]{Scheicher-Thuswaldner:03}. An implication is that
the FEP cannot hold whenever there exists some $A\neq 0\in \R$ with
$T(A)=A$.

\begin{lemma}\label{lemma3}
Assume that $(\R,X,\N)$ has the PEP. Then $(\R,X,\N)$ has the FEP if
and only if $0 \in \PP$ and $\abs{\,\PP/T\,}=1$.
\end{lemma}

\begin{proof}
Since $(\R,X,\N)$ has the PEP, it follows that the digit sequence of
each element of $\R$ ends up periodically, i.e. for every $A\in\R$
there exists $n\in\NN$ with $T^n(A)\in \PP$. Now if
$\abs{\,\PP/T\,}=1$ and $0 \in \PP$, it follows that for each $A \in
\R$ there exists an $n \in \NN$ with $T^n(A)=0$.

To show the converse, note that the orbits in $\PP/T$ are pairwise
disjoint. The requirements $0 \in \PP$ and $\abs{\,\PP/T\,}=1$ are
therefore equivalent to $0 \in \mathcal{O}$ for all orbits
$\mathcal{O} \in \PP/T$. Suppose that there exists an orbit
$\mathcal{O} \in \PP/T$ with $0 \not\in \mathcal{O}$. Then for each
element $A \in \mathcal{O}$ we have $T^n(A) \not=0$ for all $n \in
\NN$ and therefore $(\R,X,\N)$ cannot have the FEP.
\end{proof}

\begin{example}\label{ex2}
Let $\E=\FF_2[y]$, $P(x)=(y+1)x^2+yx+(y^2+1) \in \E[x]$ and
$\R=\E[x]/(P)$. Digit systems of this kind have been investigated in
\cite{Scheicher-Thuswaldner:03a}. A system of representatives of
$\R$ is, for example, $\N:=\{1,y,y+1,y^3+y\}$. Note that this set
does not include $0$. It can be easily verified that
$(y^3+y,1,1,1,y+1)$ is the zero cycle of $(\R,X,\N)$; thus we have
$$0=(y^3+y)+X+X^2+X^3+(y+1)X^4.$$
We immediately see that $\PP$ includes the orbit of $0$, which
consists of the elements
$$
0\rightarrow (y^2+y)X+y^2\rightarrow (y+1)X+y^2\rightarrow
(y+1)X+1\rightarrow y+1\rightarrow0.
$$
When $(\R,X,\N)$ has the PEP and $\PP$ consists only of these
elements, then $(\R,X,\N)$ has also the FEP. We come back to this
example in Section~\ref{FiniteFields}.
\end{example}

Several important properties of the digit system $(\R,X,\N)$ can be
derived from the constant and leading coefficients of the defining
polynomial $P$. First, we examine the pathological case when the
constant coefficient of $P$ is a unit.

\begin{lemma}\label{lemma4}
Let the polynomial $P$ be such that the constant coefficient $p_0$
is a unit of $\E$, and suppose that $(\R,X,\N)$ has the FEP. Then
$\R$ is finite.
\end{lemma}

\begin{proof}
For $p_0$ a unit, the digit set $\N$ contains exactly one element
$d$, and all elements of $(\R,X,\N)$ have the same digit sequence
$(d,d,d,\ldots)$. Thus $\PP=\R$. By the Lemma~\ref{lemma3}, $\PP$ is
finite.
\end{proof}

The last lemma may seem curious. In order to illustrate it, we will
consider the example of base $1$ in $\ZZ$, that is, we take $\E=\ZZ$
and $P=x-1$, getting $\R\cong \E$. If the unique digit is chosen to
be $0$, then clearly only $0$ has a finite expansion. If the digit
is, say, $1$, then it is true that all positive integers can be
finitely expanded, in the way of (primitive) Roman numerals or tally
marks, but negative integers have no finite expansion. If we change
the base ring to $\E=\FF_p$ for some prime $p$, and take $P$
irreducible, the tally marks do cover all elements of $\R$. In
general, if $\R$ is a finite ring and $X$ some non-zero element of
$\R$, it is not trivial to decide if every element of a finite ring
$\R$ can be written in the form $\sum_{i=0}^\ell d X^i$, for some
fixed ``digit'' $d$. The problem is related to the theory of linear
congruential sequences; see also \cite[Lemma
4.13]{VanDeWoestijne:09} and \cite[Section 3.2.1]{Knuth:98}.

Because of the strange properties of digit systems with
$\abs{\N}=1$, we will assume that $p_0$ is not a unit for the rest
of the paper.

The other important coefficient of the defining polynomial $P$ is
the leading coefficient $p_d$, which is usually taken to be $1$,
because of the better algebraic properties of the quotient ring $\R$
in that case. Several monic cases, {\it i.e.}, when $p_d$ is a unit,
have already been investigated: $\E=\ZZ$ and
$\N=\{0,\ldots,\abs{p_0}-1\}$ give the well analysed canonical
number systems (see for instance~\cite{Petho:91}). If $\E$ is the
ring of polynomials over a finite field $\FF$ we obtain the digit
systems presented in~\cite{Scheicher-Thuswaldner:03a}. Both concepts
will be generalised in the next section.

\end{section}

\begin{section}{Digit systems in the non-monic case} \label{sec2}

Most of the investigations on digit systems in the literature have
been limited to the case where $P$ is monic. One of the reasons for
this constraint is that the structure of  $\R=\E[x]/(P)$ regarded as
an $\E$-module is much more complicated if $P$ is not monic. In
fact, if the leading coefficient $p_d$ of $P$ is a unit in $\E$,
then $\R$ is a free $\E$-module of finite rank $d$; for example, the
powers $1,X,\ldots,X^{d-1}$ of $X$ form a basis. If $p_d$ is not a
unit, then $\R$ is no longer free, and it is not even finitely
generated over $\E$. We do have the following standard
representation of its elements, relative to the choice of
representatives of $\E$ modulo $p_d$.

\begin{lemma} \label{LemTail}
Let $M\subset \E$ be a set of representatives of $\E/(p_d)$. For each $A \in
\R$ there exists a unique $A'\in \E[x]$ with $\deg A'<d$ and unique
$r_d,\ldots,r_k \in M$, with $k \in \NN$ minimal, such that
\[
   A \equiv A' + \sum_{i=d}^{k} r_i x^i \bmod{P}.
\]
\end{lemma}
\begin{proof}
Let $A \in \R$ be represented by
$$f =\sum_{i=0}^k b_ix^i \in
\E[x]$$ with $k$ minimal. If $k \geq d$, there is a unique $r_k \in
M$ such that $b_k=r_k+q_k p_d$, for some $q_k\in\E$, and it follows
that
$$
f = q_k Px^{k-d} + r_k x^k + f' \equiv r_k x^k + f' \pmod{P}
$$
where $f'\in\E[x]$ has lower degree than $f$. Continuing by induction, we find
$$
  f = QP + \sum_{i=d}^k r_i x^i + \tilde{f},
$$
with $Q\in\E[x]$,  $r_i\in M$, and  $\tilde{f}\in\E[x]$ a polynomial
of degree less than $d$. In order to prove unicity, suppose that $A$
is also represented by $$g=\sum_{i=d}^k r_i' x^i + \tilde{g},$$
where $\tilde{g}\in\E[x]$ has degree less than $d$. Thus, the
difference $(f-QP)-g$ is a multiple of $P$.  We have $r_k'=r_k$,
because both are congruent modulo $p_d$, and both are in $M$.
Continuing by induction, we find that $\tilde{f}=\tilde{g}$, as
desired.
\end{proof}

The question whether a given digit system $(\R,X,\N)$ has periodic
representations or even finite expansions can already be decided by
looking at a properly chosen finitely generated $\E$-submodule of
$\R$.

\begin{theorem} \label{ThmBru1}
Let $(\R,X,\N)$ be a digit system with $\R=\E[x]/(P)$ and let
$k\geq\deg P$ be minimal such that all digits in $\N$ can be
represented by polynomials in $\E[x]$ of degree at most $k$. Let
$\R_k$ be the submodule of $\R$ generated by $X^i$ for
$i=0,\ldots,k-1$. Then $(\R,X,\N)$ has the FEP (PEP, resp.) if and
only if every element of $\R_k$ has a finite $X$-ary expansion
(periodic digit sequence, resp.).
\end{theorem}

\begin{proof}
Let $A\in\R$ be represented by $f\in\E[x]$. We will investigate the
action of $T$ on $A$. Let $c\in\E[x]$ be a representative of minimal
degree of $D_\N(A)$. Thus, $f-c$ is divisible by $X$ modulo $P$.
Hence, for some $Q\in\E[x]$, the polynomial $f - c - QP$ has zero
constant coefficient, and in fact we may take $Q$ to be constant, so
that the degree of the disturbing term $c+QP$ is bounded by $k$.
Thus, if $f$ has degree at most $k-1$, then the same holds for the
representative $\frac{f-c-QP}X$ of $T(A)$, so $T(A)\in\R_k$; whereas
if $\deg f \geq k$, then $\frac{f-c-QP}X$ has smaller degree than
$f$, so that $T^i(A)$ will be in $\R_k$ if $i$ is large enough.
Thus, it is enough to consider the finite expansion (the periodic
expansion, resp.) of the elements of $\R_k$.
\end{proof}

\begin{example}
Consider $P(x)=2x+3\in\ZZ[x]$, which corresponds to numeration with
basis $-3/2$. If all digits have degree $0$ or $1$, then we can take
$k=1$ in the theorem, and it suffices to consider $R_1$, the set of
polynomials having degree $0$ representatives. Of course, $R_1$ is
isomorphic to $\ZZ$. If some digit can only be represented by
polynomials of degree at least $m>1$, then we have $k\geq m$. Such
classes exist: as soon as $P$ is not monic, the monomials $x^m$, for
$m\ge 0$, cannot be further reduced modulo $P$.
\end{example}

\begin{remark}
Note that, in general, when $k\leq\deg P$, the module $R_k$ is free,
because we assume that the leading coefficient $p_d$ of $P$ is not a
zero divisor. This case occurs for $k=\deg P$ when all digits have
low-degree representatives. If $p_d$ is a unit in $\R$, then for any
$k$ the module $R_k$ is contained in the free $\E$-module generated
by $\frac{X^i}{p_d^k}$ for $i=0,\ldots,k$; this follows from the
normal form given in Lemma~\ref{LemTail}.
\end{remark}

Establishing the FEP for a given digit system is in general a
difficult problem. Therefore, if we have a critical subset as in the
theorem, we will be interested in making it as small as possible, so
that we need to check as few elements as possible for
representability. Below, we will show that a special independent set
in $\R$ called the \emph{Brunotte basis} (cf. \cite{Brunotte:01})
generates a rather small critical submodule of $\R$, and furthermore
brings the dynamic mapping $T$ into an especially simple form. The
Brunotte basis has already been used in the case where $P$ is monic,
or, in other words, in the CNS case (cf. \cite{Brunotte:01,
Scheicher-Thuswaldner:03}), where it is a basis of $\R$ over $\E$.
We will now generalise this concept to the non-monic case.

Because the effect of choosing this basis on the backward division
algorithm is only visible when the digit set $\N$ is chosen to be a
subset of $\E$, we will assume for the remainder of this section
that each digit has a representative contained in $\E$ and we will
identify the digits with these representatives.

\begin{definition}\label{wis}
Let $w_0=p_d$ and $w_k = Xw_{k-1} + p_{d-k}$ for $k=1,\ldots,d-1$;
then $(w_0,\ldots,w_{d-1})$ is called the {\it Brunotte basis} of
$\E[x]$ modulo $P$. The $\E$-submodule $\Lambda_P$ of $\R$ generated
by the $w_i$ will be called the \emph{Brunotte module} of $P$.
\end{definition}

We now state some easy properties of the Brunotte basis. The proofs
are straightforward and will be omitted.

\begin{lemma}
Let $(w_0,\ldots,w_{d-1})$ be the Brunotte basis of $\E[x]$ modulo
$P$.
  \begin{romanlist}
     \item
       For $i=0,\ldots,d-1$, the basis element $w_i$ is exactly the integral
       (polynomial) part of $P/X^{d-i}$. In particular, we have
       $Xw_{d-1}+p_0=P$.
     \item
       The coordinate matrix of the $w_i$, with respect to the basis
       $1,X,\ldots,X^{d-1}$, is upper triangular, and all diagonal elements are
       equal to $p_d$.
  \end{romanlist}
\end{lemma}

We have the inclusions
$$
  \Lambda_P \; \subseteq \; \bigoplus_{i=0}^{d-1} \E X^i = \R_{d} \subseteq
  \R.
$$
The quotient $\R_d/\Lambda_P$ is isomorphic to $(\E/(p_d))^d$ and
thus, $\R_d=\Lambda_P$ if and only if $p_d$ is a unit. Under the
same condition, $\R_d=\R$.

Using the Brunotte basis, we modify the representation of elements
in $\R$ given in Lemma~\ref{LemTail}.
\begin{lemma} \label{LemBruTail}
Let $M\subset \E$ be a set of representatives of $\E/(p_d)$. For
each $A \in \R$ there exist unique $q_0,\ldots,q_{d-1}\in\E$ and
unique $r_0,\ldots,r_k \in M$, with $k \in \NN$ minimal, such that
\[ A =   \sum_{i=0}^{d-1} q_iw_i+\sum_{i=0}^{k} r_i X^i.\]
\end{lemma}
\begin{proof}
Let $A\in\R$, and reduce it to the form $A'+\sum_{i=d}^k r_iX^i$, as
in Lemma \ref{LemTail}, with $\deg A'<d$. Now observe that with
respect to the usual power basis $1,X,\ldots,X^{d-1}$, the elements
$w_i$ of the Brunotte basis have leading coefficient $p_d$. We use
the $w_i$ instead of $P$ to reduce the coefficients of $f$ further,
from $w_{d-1}$ down to $w_0$, and we find
$$
  A' = \sum_{i=0}^{d-1} q_iw_i+r_{d-1} X^{d-1} + \cdots + r_1X + r_0 ,
$$
as desired.
\end{proof}

\begin{definition}
Let $M\subset \E$ a set of representatives of $\E/(p_d)$.
For an element $A \in \R$ the representation
\[
A= \sum_{i=0}^{d-1} q_iw_i+r
\]
from the above lemma is called the {\it standard representation} of
$A$ with respect to $M$. We say that $r = \sum_{i=0}^k r_iX^i$ is
the {\it residue polynomial} of $A$.
\end{definition}

The main result of this section is the following theorem, which
shows that we may take $\Lambda_P$ instead of $\R_d$ in Theorem
\ref{ThmBru1}.

\begin{theorem}\label{kern}
Let $(\R,X,\N)$ be a digit system with $\N\subseteq\E$. Then
$(\R,X,\N)$ has the FEP (PEP, resp.) if and only if each $A \in
\Lambda_P$ has a finite $X$-ary expansion (periodic digit sequence,
resp.). On $\Lambda_P$, the dynamic mapping $T$ takes the form
\begin{equation} \label{EqTBrunotte}
  T : \E^d \rightarrow \E^d,\quad (a_0,\ldots,a_{d-1}) \mapsto
      \left(a_1,\ldots,a_{d-1},
            -\frac{\sum_{i=0}^{d-1} a_ip_{d-i} - e_0}{p_0}\right),
\end{equation}
where a general element $\sum_{i=0}^{d-1} a_iw_i\in\Lambda_P$ is
represented by its coordinate sequence with respect to the basis
$(w_i)$, and $e_0\in\N$ is the unique digit that ensures
divisibility by $p_0$.
\end{theorem}
\begin{proof}
Let $A\in\R$, let $M\subset \E$ be a set of representatives of $\E/(p_d)$, and
let
\[
  A=\sum_{i=0}^{d-1} a_i w_i+ r
\]
with $r=r_k X^k+\cdots+r_1X+r_0$, $r_i \in M$, be the standard
representation of $A$. We will investigate the action of $T$ on $A$.
Let  $e_0\in\E$ represent the digit $D_\N(A)$.  Taking
\begin{equation}\label{qstern}
q=\frac{\sum_{i=0}^{d-1}a_i p_{d-i}+r_0-e_0}{p_0},
\end{equation}
we can write
\[
  A=e_0 +q p_0 + \sum_{i=0}^{d-1} a_i (w_i-p_{d-i})+ X\frac{r-r_0}X
  \in
  \E[X].
\]
Observe that $w_0-p_d=0$, and $w_i-p_{d-i}=Xw_{i-1}$ for
$i=1,\,\ldots,\,d-1$; furthermore, we have $q p_0 \equiv
-qXw_{d-1}\pmod{P}$. Set $a_d=-q$. Then the standard representation
of $T(A)$ is
\begin{equation} \label{EqTA}
  T(A) = \frac{A-e_0}X = \sum_{i=0}^{d-1} a_{i+1}w_i + \frac{r-r_0}X.
\end{equation}
In other words: after one application of $T$ the degree of the residue
polynomial $r$ decreases by one and the first $k$ coefficients do not change.
Hence, after $k+1$ applications, we have
\[
  T^{k+1}(A) \in \Lambda_P.
\]
The first assertion of the theorem follows immediately. The second
assertion follows from \eqref{EqTA} by taking $r=r_0=0$ and taking
$a_d=-q$ from (\ref{qstern}).
\end{proof}

We end the section with the special case where the base ring $\E$ is
Euclidean. Here we obtain a necessary condition for the finite expansion
property that is analogous to the usual expanding property of $P$.

\begin{theorem}
Let $\E$ be Euclidean with value function $g:\E \mapsto \RR^+ \cup
\{0,-\infty\}$ where $g(0)=-\infty$, and let $(\R,X,\N)$ be a digit
system satisfying $\N\subset\E$ and $g(e)<g(p_0)$ for all $e \in
\N$. If $g(p_d) \geq g(p_0)$, then no element of $\Lambda_P$  but
$0$ has a finite $X$-ary expansion.
\end{theorem}
\begin{proof}
First note that the assumption on $\N$ implies $0 \in \N$. Let
$\pi:\E[x] \rightarrow \R$ be the canonical epimorphism and $A \in
\E[x]$. Because the leading coefficient of $w_k$ is $p_d$ for each
$k \in \{0,\ldots,d-1\}$, it is easy to see that $\pi(A) \in
\Lambda_P$ implies that the leading coefficient of $A$ is a multiple
of $p_d$. Now suppose that there is a $B \in \Lambda_P$, $B \not=0$
with finite $X$-ary expansion
\[
  B=\sum_{i=0}^h e_iX^i,\quad e_i \in \N, e_h \not=0.
\]
By assumption we have that
$g(e)<g(p_0)$ for $e \in \N$ and therefore $g(e_h)<g(p_0) \leq g(p_d)$. As
observed above we also must have that $e_h=qp_d$ for some nonzero $q \in \E$.
But $\E$ is Euclidean, so $q,p_d \not=0$ implies $g(e_h)=g(qp_d)\geq g(p_d)$,
which is a contradiction.
\end{proof}
\begin{corollary}\label{1char}
With the above assumptions on $\E$ and $\N$, $g(p_d) < g(p_0)$ is necessary for $(\R,X,\N)$ to have the finite expansion property.
\end{corollary}
\begin{example}
We retrieve the following known result for a linear polynomial
$P(x)=p_1x+p_0$ with $p_1\ne0$: if $\E=\ZZ$ and $|e|<|p_0|$ holds
for all $e\in\N$, then Corollary~\ref{1char} tells us that
$\abs{p_0} > \abs{p_1}$ is necessary for $(\ZZ[x]/(P),X,\N)$ to have
finite expansions. Of course, $\abs{p_0}
> \abs{p_1}$ is equivalent to the polynomial $P$ being expanding.
\end{example}

\end{section}

\begin{section}{Products of digit systems} \label{sec3}

Consider the exact sequence
$$
  0\rightarrow \E[x]/(P_2) \xrightarrow{\cdot P_1} \E[x]/(P_1P_2) \rightarrow
  \E[x]/(P_1) \rightarrow 0,
$$
of $\E$-modules, and suppose we have defined number systems on the
outer components. Below, we construct a naturally defined number
system on the middle module and derive conditions when all elements
in this number system have a finite expansion, or at least a
periodic digit sequence.

\begin{theorem}\label{multgcns}
Let
$$
P_1(x)=p_0+p_1x+\cdots+p_m x^m,\quad P_2(x)=p'_0+p'_1x+\cdots+p'_{n}
x^{n},
$$
and for $i\in\{1,2\}$, let $\R_i=\E[x]/(P_i)$ and $\N_i$ be systems
of representatives of $\R_i/(X)$ with $\N_i\subset\E$. Let $
\mathcal{M}=\{d+eP_1 :d\in\mathcal{N}_1,e\in\mathcal{N}_2\}. $ Then
the following assertions hold:
\begin{itemize}
\item $(\E[x]/(P_1P_2),X,\mathcal{M})$ is a digit system.
\item if $0\in\N_1$ and $(\R_1,X,\N_1)$ and
$(\R_2,X,\N_2)$ have the FEP, then $(\E[x]/(P_1 P_2),X,\mathcal{M})$
has the FEP.
\item if $0\in\N_1$, $(\R_1,X,\N_1)$ has the FEP, and
$(\R_2,X,\N_2)$ has the PEP, then $(\E[x]/(P_1 P_2),X,\mathcal{M})$
has the PEP.
\end{itemize}
\end{theorem}
\begin{proof} We only prove the second assertion, as the first is
trivial, and the third follows analogously to the second.

Consider an element
$$
A=a_0+a_1 x+\cdots+a_{N}x^{N}\in\mathcal{E}[x].
$$
Let $d_0\in\mathcal{N}_1$, $e_0\in\mathcal{N}_2$,
$k,\ell\in\mathcal{E}$ such that
$$
a_0=d_0+k p_0\qmboxq{and} k=e_0+\ell p'_0.
$$
Let $A^{(0)}=A$ and $B^{(0)}=0$. Since
$$
\begin{array}{rll}
  p_0 & \equiv-p_1 X-\cdots-p_m X^m+P_1  &\pmod{P_1 P_2}\quad\mbox{and}\\
P_1 p'_0 & \equiv(-p'_1 X-\cdots-p'_n X^n)P_1 &\pmod{P_1 P_2},
\end{array}
$$
we obtain that
\begin{align*}
A &=  A^{(0)}+P_1 B^{(0)}\\
&\equiv  d_0+
(a_1-k p_1)X+\cdots+(a_m-k p_m)X^{m}+a_{m+1}X^{m+1}+\cdots+a_N X^N+\\
&\quad  +k P_1\\
&\equiv  d_0+
(a_1-k p_1)X+\cdots+(a_m-k p_m)X^{m}+a_{m+1}X^{m+1}+\cdots+a_N X^N+\\
&\quad  +e_0 P_1+(-\ell p'_1 X-\cdots-\ell p'_n X^n)P_1\\
&\equiv  d_0+e_0 P_1+X (A^{(1)}+P_1 B^{(1)})\pmod{P_1 P_2}
\end{align*}
with $d_0+e_0 P_1 \in\mathcal{M}$. Iterating this process, we get a
recurrence for the coefficients of
$$
A^{(j)}=\sum_{i\geq0}a_i^{(j)}X^i\quad\qmboxq{and}\quad
B^{(j)}=\sum_{i\geq0}b_i^{(j)}X^i.
$$
Starting with
$$
a_i^{(0)}= \left\{
\begin{array}{rcl}
a_i   &\mbox{for}&  i\leq N, \\
0     &\mbox{for}&  i>N,
\end{array}
\right. \quad\quad b_i^{(0)}=0\qmboxq{for}i\geq0
$$
we obtain for $j\geq0$,
$$
\renewcommand{\arraystretch}{2.0}
\begin{array}{rcl}
d_j&\equiv&a_{0}^{(j)},\\
e_j&\equiv&b_{0}^{(j)}+k^{(j)},
\end{array}
\quad
\begin{array}{rcl}
k^{(j)}&=&(a_0^{(j)}-d_j)/{p_0},\\
\ell^{(j)}&=&(b_0^{(j)}+k^{(j)}-e_j)/{p'_0},
\end{array}
\quad
\begin{array}{rcl}
a_i^{(j+1)}&=&a_{i+1}^{(j)}-k^{(j)} p_{i+1},\\
b_i^{(j+1)}&=&b_{i+1}^{(j)}-\ell^{(j)} p'_{i+1},
\end{array}
$$
with $d_j\in\mathcal{N}_1$, $e_j\in\mathcal{N}_2$. Then
$$
A^{(j)}+P_1 B^{(j)} =d_j+e_j P_1+ X ( A^{(j+1)}+P_1 B^{(j+1)} ).
$$
Note that the recurrence for $k^{(j)}$, $a_i^{(j)}$ and $d_j$ is
just backward division algorithm for $A$, considered as an element
of $(\R_1,X,\N_1)$. Since $(\R_1,X,\mathcal{N}_1)$ has the FEP,
there is an index $j_0$ such that $k^{(j)}=0$ for all $j\geq j_0$.
Then, for $j\geq j_0$, the recurrence for the $b^{(j)}_i$ is no
longer disturbed by the $k^{(j)}$. Since $(\R_2,X,\mathcal{N}_2)$
has the FEP , there is an $j_1\geq j_0$ such that $\ell^{(j)}=0$ for
all $j\geq j_1$. Thus, we obtain a finite expansion
$$
A\equiv\sum_{j\geq0}(d_j+e_j P_1)X^j.
$$
\end{proof}

It is possible (but tedious) to extend this theorem to the case
where the $\N_i$ are no longer assumed to be subsets of $\E$.
Assuming this has been done, the following generalisation of
Theorem~\ref{multgcns} follows by induction on the number of
factors.

\begin{corollary}
For $i=1,\ldots,k$, let $(\E[x]/(P_i),X,\mathcal{N}_i)$ be digit
systems. Let
$$
\mathcal{M}= \{d_1+d_2 P_{1}+ d_3 P_1P_2 + \cdots+d_kP_{1}\cdots
P_{k-1}\}\qmboxq{with} d_i\in\mathcal{N}_{i}.
$$
Then $(\E[x]/(P_1\cdots P_k),X,\mathcal{M})$ is a digit system. If
$(\E[x]/(P_i),X,\N_i)$ has the FEP for all $i$ and $0\in\N_i$ for
$1\le i\le k-1$, then $(\E[x]/(P_1\cdots P_k),X,\mathcal{M})$ also
has the FEP.
\end{corollary}

Note that Theorem~\ref{multgcns} and its corollary are asymmetric.
Thus, by permutation of the $P_i$, one can obtain other digit
systems.
\end{section}

\begin{section}{Generalisation of known digit systems} \label{sec4}

Several known families of digit systems have the form
$(\E[x])/(P),X,\N$ for a monic polynomial $P$, e.g. taking
$\E=\mathbb{Z}$ or $\E=\FF_q[y]$. These families will now be
generalised to the case where $P$ is not necessarily monic.
\begin{subsection}{Generalisation of canonical number systems}\label{CNS}
Assume that $\E=\ZZ$.  If $P\in\ZZ[x]$ is monic and we choose
$\N=\{0,\ldots,\abs{p_0}-1\}$, then $(\R,X,\N)$ has the finite
expansion property if and only if the pair $(P,\N)$ is a canonical
number system (CNS) in the sense of \cite{Petho:91}. What we have
done so far, allows us to extend this definition as follows.

\begin{definition}
Let $P=p_dx^d+\cdots
+p_0 \in \ZZ[x]$, $\R=\ZZ[x]/(P)$ and $\N=\{0,\ldots,\abs{p_0}-1\}$.
If $(\R,X,\N)$ has the finite expansion property, then we call
$(P,\N)$ a {\em canonical number system} (CNS).
\end{definition}

It is an open problem to characterise all CNS, even in the monic
case. Many partial results have been obtained (see
\cite[Section~3.1]{BBLT2006} and the literature given there),
several of which immediately generalise to the non-monic case. One
of the most promising directions here is the characterisation of CNS
in terms of the discrete dynamical systems called shift radix
systems (SRS).

\begin{definition}[first given in \cite{Akiyama-Borbeli-Brunotte-Pethoe-Thuswaldner:05}] \label{DefSRS}
Let $d$ be a positive integer. For a $\mathbf{r} \in \RR^d$ define
the function
\begin{align*}
  \tau_\mathbf{r}:\ZZ^d & \rightarrow \ZZ^d ,\\
  \mathbf{z}=(z_0,\ldots,z_{d-1}) &
    \mapsto (z_1,\ldots,z_{d-1},-\fl{\mathbf{rz}}),
\end{align*}
where $\mathbf{rz}$ is the scalar product of $\mathbf{r}$ and
$\mathbf{z}$. The mapping $\tau_\mathbf{r}$ is called a {\em shift
radix system} (SRS) if for any $\mathbf{z} \in \ZZ^d$ there exists a
$k \in \NN$ such that the $k$-th iterate of $\tau_{\bf r}$ satisfies
$\tau_\mathbf{r}^k(\mathbf{z}) = \mathbf{0}$. Define the sets
\begin{align*}
  \D_{d}:= & \left\{{\bf r} \in \RR^{d} \mid \mbox{each orbit of } \tau_{\bf r} \mbox{ is ultimately
    periodic} \right\} \mbox{ and} \\
  \D_{d}^{(0)}:= & \left\{{\bf r} \in \RR^{d} \mid \tau_{\bf r} \mbox{ is an SRS}
  \right\}.
\end{align*}
\end{definition}
The main result relating CNS with SRS  is proved for monic
polynomials in \cite[Theorem
3.1]{Akiyama-Borbeli-Brunotte-Pethoe-Thuswaldner:05}. It can be
generalised also to non-monic polynomials as follows.
\begin{theorem}
\label{ThmSRSLink} Let $P(x)=p_dx^d+\cdots+p_0 \in \ZZ[x]$,
$\mathcal{R}=\ZZ[x]/(P)$, $\N=\{0,\ldots,\abs{p_0}-1\}$ and $$ {\bf
r}=\left(\tfrac{p_d}{p_0},\tfrac{p_{d-1}}{p_0},\ldots,
\tfrac{p_1}{p_0}\right). $$ Then the following assertions hold:
\begin{itemize}
\item $(\mathcal{R},X,\mathcal{N})$ has the FEP (i.e., $(P,\N)$ is a
CNS) if and only if ${\bf r}\in \D_d^{(0)}$.
\item $(\mathcal{R},X,\mathcal{N})$ has the PEP if and only if ${\bf r}\in
\D_d$.
\end{itemize}
\end{theorem}
\begin{proof}
As in Definition~\ref{wis}, let $\{w_0,\ldots,w_{d-1}\}$ be the
Brunotte basis of $\E[x]$ modulo $P$. One easily computes that
$w_k=\sum_{i=0}^k p_{d-i}X^{k-i}$. By Theorem~\ref{kern},
$(\mathcal{R},X,\mathcal{N})$ has the FEP if and only if each
element $A$ of the shape
\[
  A=A^{(0)}=\sum_{i=0}^{d-1}a_iw_i, \quad a_i \in \ZZ,
\]
has a finite $X$-ary expansion.
In the proof of Theorem~\ref{kern} we showed that an application of $T$ yields
$A^{(0)}=e_0 +XA^{(1)}$ with
\[
  A^{(1)}=T(A^{(0)})=\sum_{i=0}^{d-1} a_{i+1}w_i
\]
and $e_0=D_\N(A^{(0)}) \in \N$. In order to find a value of $a_d$,
we apply \eqref{qstern}; by the form of the digit set $\N$ we find
\[
  a_d=-q=-\fl{{\bf r} \cdot (a_0,\ldots,a_{d-1})}.
\]
Thus $A^{(1)}=(w_0,\ldots,w_{d-1})\cdot
\tau_{\mathbf{r}}\!\left((a_0,\ldots,a_{d-1})\right)$. Recall that
$(\mathcal{R},X,\mathcal{N})$ has the FEP if and only if successive
application of $T$ to each $A \in \Lambda_P$ ends up in $0$, and we
see that this is equivalent to $\tau_{\mathbf{r}}$ being an SRS. The
second assertion can be shown analogously.
\end{proof}

The connection just given allows for an easy proof of the following
criterion, which was given for the monic case in
\cite[Proposition~7]{Gilbert:81} and \cite{Kovacs:81a}.

\begin{theorem}
  Let $P=p_dx^d+\cdots+p_1x+p_0 \in \ZZ[x]$ and $\N=\{0,\ldots,\abs{p_0}-1\}$.
  Suppose
  $$
   p_0\geq2\quad\mbox{and}\quad p_0 \ge p_1 \ge \ldots \ge p_d > 0.
  $$
  Then $(P,\N)$ is a CNS.
\end{theorem}

\begin{proof}
Theorem~3.5 of \cite{Akiyama-Brunotte-Pethoe-Thuswaldner:06} tells
us that if $0\leq r_1\leq\ldots\leq r_d<1$ then
$\mathbf{r}\in\D_d^{(0)}$. Therefore, the result follows directly
from Theorem~\ref{ThmSRSLink}].
\end{proof}


In \cite{Akiyama-Scheicher:07} so-called symmetric canonical number
systems (SCNS) have been introduced. The difference to usual CNS is
the digit set: it is almost symmetric around $0$, {\em viz.}
$$\N=\left[-\frac{\abs{p_0}}{2},\frac{\abs{p_0}}{2}\right) \cap
\ZZ.$$ In \cite{Surer:09} the definition has been generalised to any
digit set consisting of $\abs{p_0}$ consecutive integers including
$0$. We will give these definitions in our formalism, generalised to
not necessarily monic polynomials.
\begin{definition}[{{\it cf.}~\cite{Akiyama-Scheicher:07,Surer:09}}]
Let $\varepsilon\in [0,1)$, $P=p_dx^d+\cdots
+p_0 \in \ZZ[x]$, $\R=\ZZ[x]/(P)$ and
$\N:=[-\varepsilon\abs{p_0},(1-\varepsilon)\abs{p_0}) \cap \ZZ$. If
$(\R,X,\N)$ has the finite expansion property, then we call $(P,\N)$
an $\varepsilon$-canonical number system ($\varepsilon$-CNS). A
$\frac{1}{2}$-CNS is also called symmetric canonical number system
(SCNS).
\end{definition}
Analogously to Theorem~\ref{ThmSRSLink},  $\eps$-CNS are closely
related to a slight modification of SRS.

For an $\varepsilon \in [0,1)$ and an $\mathbf{r} \in \RR^d$ let
\[\begin{split}
\tau_{\mathbf{r},\varepsilon}: \, & \ZZ^d  \rightarrow \ZZ^d, \\
& \mathbf{z}=(z_0,\ldots,z_{d-1}) \mapsto
(z_1,\ldots,z_{d-1},-\fl{\mathbf{rz}+\eps}).
\end{split}\]
Define
\begin{align*}
  \D_{d,\eps}:= & \left\{{\bf r} \in \RR^{d} \mid \mbox{each orbit of } \tau_{\bf r,\eps} \mbox{ is ultimately
    periodic} \right\} \mbox{ and} \\
  \D_{d,\eps}^{(0)}:= & \left\{{\bf r} \in \RR^{d} \mid \mbox{each orbit of } \tau_{\bf r,\eps} \mbox{ is ultimately
    zero}
  \right\}.
\end{align*}
Note that $\D_d=\D_{d,0}$ and $\D_d^{(0)}=\D_{d,0}^{(0)}$. The
following theorem can be proved in a similar way as
Theorem~\ref{ThmSRSLink}.
\begin{theorem}[{{\it cf.}~\cite{Akiyama-Scheicher:07,Surer:09}}] \label{ThmeSRSLink}
Let $P(x)=p_dx^d+\cdots+p_0 \in \ZZ[x]$, $\mathcal{R}=\ZZ[x]/(P)$,
$\varepsilon\in [0,1)$,
$\N=[-\varepsilon\abs{p_0},(1-\varepsilon)\abs{p_0}) \cap \ZZ$ and
$$ {\bf r}=\left(\tfrac{p_d}{p_0},\tfrac{p_{d-1}}{p_0},\ldots,
\tfrac{p_1}{p_0}\right). $$  Then  the following assertions hold:
\begin{itemize}
\item $(\mathcal{R},X,\mathcal{N})$ has
the FEP (i.e., $(P,\N)$ is an $\varepsilon$-CNS) if and only if
${\bf r}\in \D_{d,\varepsilon}^{(0)}$.
\item $(\mathcal{R},X,\mathcal{N})$
has the PEP if and only if ${\bf r}\in \D_{d,\varepsilon}$.
\end{itemize}
\end{theorem}

\begin{example}
We are now able to decide whether the digit systems in
Example~\ref{ex1} have the PEP or the FEP by using
Theorem~\ref{ThmSRSLink} and Theorem~\ref{ThmeSRSLink}. Set
$P(x)=3x^2-2x+5$, $\N=\{0,1,2,3,4\}$ and $\M=\{-2,-1,0,1,2\}$. Thus
we are interested in the mappings $\tau_{\bf r}$ and $\tau_{{\bf
r},\frac{1}{2}}$ corresponding to the vector ${\bf
r}:=\{\tfrac{3}{5},-\tfrac{2}{5}\}$. By
\cite[Lemma~5.2]{Akiyama-Brunotte-Pethoe-Thuswaldner:06} we have
${\bf r} \in \D_2^{(0)}$ and therefore $(\R,X,\N)$ has the FEP. On
the other hand, according to
\cite[Theorem~5.2]{Akiyama-Scheicher:07}, ${\bf r} \not \in
\D_{d,\frac{1}{2}}^{(0)}$ but ${\bf r} \in \D_{d,\frac{1}{2}}$
however. This shows that $(\R,X,\M)$ has the PEP but not the FEP.
\end{example}

\end{subsection}

\begin{subsection}{Generalisation of digit systems over a finite field}\label{FiniteFields}

Let $\E=\FF[y]$ be the ring of polynomials in $y$ over the finite
field $\FF$. It is well-known that $\FF[y]$ is a Euclidean domain.
The corresponding value function $g$ assigns to an element $q$ of
$\FF[y]$ its degree, which we denote by $\deg_y(q)$.

We choose a defining polynomial $P=p_dx^d+\cdots+p_0 \in \FF[y][x]$
and put $\R=\FF[y][x]/(P)$. For the digit set $\N$, the canonical
choice is
\[\N=\{q \in \FF[y]\mid \deg_y(q) < \deg_y{p_0}\}.\]
This gives us a digit system $(\R,X,\N)$. Scheicher and
Thuswaldner~\cite{Scheicher-Thuswaldner:03a} analysed the case where
$P$ is monic in $y$, {\it i.e.,}\ where the leading coefficient
$p_d$ is a nonzero element of $\F$. In their terminology, the pair
$(P,\N)$ is a {\it digit system} if $(\R,X,\N)$ has the finite
expansion property. Their main result is the observation that
$(P,\N)$ is a digit system if and only if $\max_{i =1}^{d-1}
\deg_y(p_i) < \deg_y(p_0)$. We will generalise this statement for
any choice of $p_d$.

\begin{theorem}[{{\it cf.}~\cite[Theorem 2.5]{Scheicher-Thuswaldner:03a}}]\label{ff}
Let $\FF$ be a finite field,  $P(x)= p_dx^d+\cdots+p_1x+p_0 \in
\FF[y][x]$, $\R=\FF[y][x]/(P)$ and $\N=\{q \in \FF[y]\mid \deg_y(q)
< \deg_y{p_0}\}$. Then the following assertions hold:
\begin{itemize}
\item $(\R,X,\N)$
has the FEP if and only if $\max_{i=1}^d \deg_y(p_i) < \deg_y(p_0)$.
\item $(\R,X,\N)$ has the PEP if and only if
$\max_{i=1}^d \deg_y(p_i) \leq \deg_y(p_0).$
\end{itemize}
\end{theorem}

\begin{proof}
Because of Theorem~\ref{kern} it is enough to check finiteness and
periodicity for the Brunotte-module $\Lambda_P$. In \cite[Lemmas 2.2
and 2.4]{Scheicher-Thuswaldner:03a} the first assertion was shown
for monic polynomials, {\it i.e.,}\ where $\Lambda_P=\R$ (see also
the much simpler proof in \cite[Theorem~2.2]{BBST2009}). The second
assertion was shown in
\cite[Theorem~3.1]{Scheicher-Thuswaldner:03a}. These proofs can be
immediately adapted to the non-monic case, since they do not use the
monicity  condition $\deg_y(p_d)=0$.
\end{proof}
\begin{remark}
The first assertion of Theorem~\ref{ff} remains valid even if $\FF$
is not assumed to be finite (cf.~\cite{BBST2009}). However, this
case does not fit in our framework as the digit set $\N$ is infinite
in this case.
\end{remark}
\begin{example}
We will prove that the number system $(\FF_2[y][x]/(P),X,\N)$
defined in Example~\ref{ex2} has the FEP. Recall that
$P(x)=(y+1)x^2+yx+(y^2+1)\in\FF_2[y][x]$ and $\N=\{1,y,y+1,y^3+y\}$.
We cannot use Theorem~\ref{ff} here, since the digit set $\N$ has
the wrong shape. But we will show a strategy how to decide the
problem with the aid of the theorem, which has also been used
successfully in the case of number systems over $\ZZ$ (see
\cite[Section 3]{VanDeWoestijne:09}). We will show that every $A \in
\R$ has a finite $X$-ary expansion with digits in $\N$. For $A=0$ we
can take the empty expansion. For $A\ne 0$, however, we will make
use of the zero cycle, which gives us the expansion
$0=(y^3+y)+X+X^2+X^3+(y+1)X^4$. Now let $A$ be an arbitrary nonzero
element of $\R$. By Theorem~\ref{ff} we know that $A$ has a finite
$X$-ary expansion with digits in the standard digit set
$\N'=\{0,1,y,y+1\}$, say,
\[
  A=\sum_{j=0}^h b_j X^j
\]
with $b_j \in \N'$ for all $j$ and with $h$ minimal. If all of the
$b_j$ are contained in $\N$, {\it i.e.}, $b_j\ne 0$ for all $j$, we
are done. Otherwise obey the following instructions:
\begin{enumerate}
  \item
    Let $i$ be the smallest index with $b_i=0$.\label{pt1}
  \item
    Add successively $y^3+y$ to $b_i$, $1$ to $b_{i+1}, b_{i+2}$ and $b_{i+3}$,
    and $y+1$ to $b_{i+4}$, if necessary changing the value of $h$
    such that $b_h\ne0$.
  \item
    If $b_j \ne 0$ for $0\le j\le h$, then
    $A=\sum_{j=0}^h b_j X^j$ is the finite $X$-ary expansion of $A$, otherwise
    return to \eqref{pt1}.
\end{enumerate}
We immediately see that $i$ increases. Furthermore, the addition of $1$ or
$Y+1$, respectively, to an element of $\N'$ gives either $0$ or an element of
$\N$. Hence, at any moment we have $b_j \in \N$ for each $j<i$. The question is
whether the procedure terminates. Note that from the moment that $h$ starts to
increase, if it occurs, only the most significant $5$ digits are of interest,
and in fact, we can reformulate the above procedure as the iteration of a
map $\Phi:S\rightarrow S$, where $S=\{ (c_1,\ldots,c_4) : c_i \in\N\cup\N' \}$,
defined as follows:
$$
  \Phi(c_1,\ldots,c_4)=\begin{cases}
      (c_2,c_3,c_4,0) & \text{ if } c_1\ne 0; \\
      (c_2+1,c_3+1,c_4+1,y+1) & \text{ if } c_1 = 0. \end{cases}
$$
Here, the new digit $c_1$, be it changed to $y^3+y$ or unchanged, is
immediately discarded, and the $4$-digit window on the expansion
moves one step to the right. We are done when for every
$(c_1,\ldots,c_4)$ with digits in $\N'$, there exists an integer $m$
such that $\Phi^m(c_1,\ldots,c_4)=(0,0,0,0)$. This is easily
verified by computer. For example, if the original expansion ends
(on the most significant side) in $1,0,y,1$, we obtain
\begin{align*}
  (1,0,y,1)&\rightarrow (0,y,1,0)\rightarrow (y+1,0,1,y+1) \rightarrow
  (0,1,y+1,0) \rightarrow (0,y,1,y+1) \\
    &\rightarrow (y+1,0,y,y+1)
  \rightarrow (0,y,y+1,0) \rightarrow (1,y+1,y,y+1)\\
  &\rightarrow (y+1,y,y+1,0) \rightarrow (y,y+1,0,0) \rightarrow
  (y+1,0,0,0) \rightarrow (0,0,0,0).
\end{align*}
Thus, the procedure terminates in all cases and actually yields the
finite $X$-ary expansion for $A$, and the FEP for $(\R,X,\N)$ has
been established.
\end{example}

\end{subsection}

\end{section}

\begin{section}{Sets of witnesses for the finite expansion
property} \label{sec5}

In Section~\ref{sec2} above, we already gave some results to the
effect that the PEP and FEP of a digit system $(\R,X,\N)$ can be
decided by looking at the expansions of elements in certain
$\E$-submodules of $\R$. Building on these results, we will show
here that for a large class of digit systems, the PEP and FEP can be
decided by checking a finite subset only.

\begin{definition}\label{DefSOW}
Let $(\R,X,\N)$ be a digit system, and $\SSS$ an arbitrary subset of
$\R$. A set $\V \subset \SSS$ with the properties
\begin{enumerate}
\item \label{prop1}
  every $A\in \SSS$ can be written as a finite sum of elements of
  $\V$;
\item \label{prop2}
  for each $e\in \N$, the set $\V$ is closed under $A\mapsto
  T(A+e)$
\end{enumerate}
is called a {\it set of witnesses} of $\SSS$ with respect to
$(\R,X,\N)$.
\end{definition}

\begin{remark} \label{RmkGen}
As to \eqref{prop1}, when $\SSS$ is an additive subgroup of $\R$,
the condition can be satisfied by including in $\V$ a set of
generators of $\SSS$ as well as their (additive) inverses.
\end{remark}

\begin{lemma}\label{AllgBrunHelp}
Let $(\R,X,\N)$ a digit system, $\SSS$ a subset of $\R$ and $\V$ a
set of witnesses of $\SSS$ with respect to $(\R,X,\N)$. Then all
elements of $\SSS$ admit finite $X$-ary expansions if and only if
the same holds for all witnesses $v\in\V$.
\end{lemma}
\begin{proof}
The implication ``$\Rightarrow$'' is clear. Thus suppose that all
elements of $\V$ have finite $X$-ary expansions. This means that for
each element $v$ of $\V$ there exists an $n \in \NN$ with
$T^n(v)=0$.

Let $A\in\SSS$ and $v\in\V$. Then
\begin{equation}\label{thussi}
\begin{split}
  T(A+v) &= \frac{A+v-D_\N(A+v)}{X} = \frac{A-D_\N(A) + v +D_\N(A)-D_\N(A+v)}{X} \\
         &= \frac{A-D_\N(A)}{X} + \frac{v+D_\N(A)-D_\N(A+v)}{X} = T(A) +
         v',
\end{split}
\end{equation}
where $v'\in\V$ by property \eqref{prop2} of
Definition~\ref{DefSOW}.

Now suppose there exist elements in $\SSS$ having no finite $X$-ary
expansion; let $A=v_1+\cdots+v_k$ be such an element for which $k$
is minimal. Using \eqref{thussi} iteratively for $k-1$ times, we
have
$$
  T(A)=T(v_1)+v_2'+\cdots+v_k'
$$
for $v_2',\ldots,v_k'\in\V$. As $v_1\in\V$ there exists an $n$ such
that $T^n(v_1)=0$. Thus, repeating the above procedure $n$ times, we
get
$$
  T^n(A) = T^n(v_1)+v_2''+\cdots+v_k'' = v_2''+\ldots+v_k''.
$$
The element $v_2''+\cdots+v_k''$ also cannot have a finite $X$-ary
expansion, because this would imply $T^m(A)=0$ for some $m$. As
$v_2'',\ldots,v_k''\in \V$ this contradicts the minimality of $k$.
\end{proof}

As above, we write $\R_k$ for the submodule of $\R$ generated by
$1,X,\ldots,X^{k-1}$.

\begin{theorem} \label{ThmWitness1}
Let $(\R,X,\N)$ be a digit system, let $k\geq\deg P$ be minimal such
that all digits in $\N$ can be represented by polynomials in $\E[x]$
of degree at most $k$, and let $\V$ be a set of witnesses of $\R_k$.
Then $(\R,X,\N)$ has the FEP if and only if every element in $\V$
has a finite $X$-ary expansion.
\end{theorem}

\begin{proof}
  This follows directly by combining the above lemma with Theorem~\ref{ThmBru1}.
\end{proof}

The concept of witness sets is useful mainly in those situations
where we can construct a \emph{finite} witness set. This is the case
for a large class of number systems, as we will now show.

First observe that one can construct a set of witnesses by the
following iterative approach. Let $\V_0\subseteq \SSS$ be an
arbitrary subset satisfying condition \ref{prop1}. Then, for $i\ge
0$, define
\begin{equation} \label{ConstrW}
  \V_{i+1} = \V_i \cup \left\{ T(v+e) \mid v\in\V_i, e\in\N \right\},
  \text{ and }
  \V = \bigcup_{i\ge 0} \V_i.
\end{equation}

We apply this construction to the module $\R_k$ as in
Theorem~\ref{ThmWitness1}. As this module is finitely generated, we
can start off with a \emph{finite} set $\V_0$, in view of
Remark~\ref{RmkGen}. Now the question is whether the sequence of the
$\V_i$ is eventually stable, and whether the resulting set $\V$ is
finite.

An important special case where this can be proved is the case where
$\R$ can be embedded as a discrete set in a finite-dimensional
complex vector space. Here we examine whether the operator $A\mapsto
XA$ is expanding, {\it i.e.}, whether all its eigenvalues have
modulus greater than $1$. It is well known that the expanding
property implies the PEP in this case (special instances occur in
\cite{Akiyama-Rao:03,Pethoe:91}; the proof in the general case is
the same).

\begin{proposition} \label{AllgBrun}
Suppose $\R$ is embedded in a finite-dimensional complex vector
space, let $\lambda$ be the operator norm of $A \mapsto A/X$ on
$\R$, and assume that $\lambda<1$. Then in the construction
\eqref{ConstrW}, we have
$$
  \max_{v\in\V} \|v\| \le \max\left\{ \max_{v\in\V_0} \|v\|,
                                 \frac{2\lambda}{1-\lambda}
                                 \max_{e\in\N}\|e\|\right\}.
$$
In particular, if $\R$ is a discrete set and $\V_0$ is finite, then
$\V$ is finite.
\end{proposition}

\begin{proof}
We have
$$
  \|T(A+e)\| = \left\|\frac{A+e-D_\N(A+e)}{X}\right\| \le \lambda (\|A\| +
  \|e\|+\|D_\N(A+e)\|);
$$
thus
$$
  \|T(A+e)\|<\|A\| \quad\text{unless}\quad \|A\|<\frac{2\lambda}{1-\lambda}
  \max_{e\in\N}\|e\|.
$$
\end{proof}

\begin{example}
We will illustrate the usage of a set of witnesses by the
following example. Let $\E=\ZZ[i]$ be the ring of Gaussian
integers, $P(x)=(1+i)x+(1+2i) \in \ZZ[i][x]$, $\R=\ZZ[i][x]/(P)$
and $\N=\{0,1,2,3,4\}$; $\N$ is a complete set of residues of
$\R/(X)$, because they are a constant polynomials and form a
complete set of residues of $\ZZ[i]/(1+2i)$. Since
$\abs{\frac{1+2i}{1+i}}=\sqrt{\frac{5}{2}}$ we see that $P$ is
expanding and therefore $(\R,X,\N)$ has the PEP. The Brunotte
basis of $\E[x]$ modulo $P$ is composed of the single element
$w_0=1+i$. As $\N \subset \ZZ[i]$ we know by Theorem~\ref{kern}
that $(\R,X,\N)$ has the FEP if and only if each element of
$\Lambda_P=(1+i)\ZZ[i]$ has a finite $X$-ary expansion. Moreover,
by the expanding property of $P$, Proposition~\ref{AllgBrun}
implies the existence of a finite set of witnesses $\V$ for
$\Lambda_P$; we claim that
\[
\V=\{0,\pm 1 \pm i, \pm 2, \pm(3-i), \pm(4-2i), \pm(2-2i)\}
\]
is such a set. Of course, the first condition of
Definition~\ref{DefSOW} is satisfied, since $1+i$ and
$-1+i=i(1+i)$ are additive generators of the Brunotte module
$(1+i)\ZZ[i]$, and they and their negatives are in $\V$. To show
the second condition we need to check that
\[
\{ T(v + e) \mid v\in\V, e \in \N \} \subseteq \V.
\]
This can be checked by direct calculation. For example, for
$v=3-i$ we get
\[
\{ T(3 - i + e) \mid e \in \N \} = \{-1+i, -4+2i\} \subseteq \V.
\]
The other elements can be treated likewise. It remains to check
that the orbit $(T^n(v))_{n\ge 0}$ contains zero for each $v \in
\V$. In Figure~\ref{f1} the action of $T$ on the elements of $\V$
is indicated by arrows.

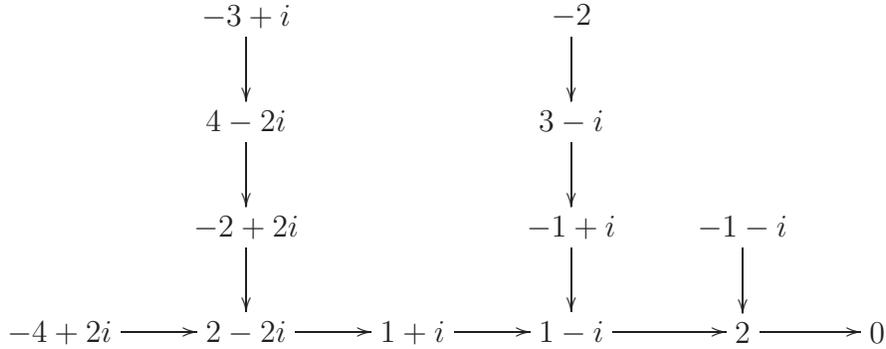
\begin{figure}[h]
\hskip 0cm \xymatrix{
&-3+i \ar[d]  &&-2\ar[d]&& \\
&4-2i \ar[d] &&3-i\ar[d]&& \\
&-2+2i\ar[d] &&-1+i\ar[d]&-1-i\ar[d]& \\
-4+2i\ar[r]&2-2i \ar[r]&1+i\ar[r]&1-i\ar[r]&2\ar[r]& 0 }
\caption{The action of $T$ in the set of witnesses $\V$\label{f1}}
\end{figure}

This shows that all orbits of $\V$ end in $0$. This shows that the
periodic set of the digit system contains just the loop at $0$. By
Lemma~\ref{lemma3}, the digit system $(\R,X,\N)$ has the FEP.
\end{example}

\end{section}

\bibliographystyle{siam}
\bibliography{nmv}

\end{document}